\documentclass[11pt]{amsart}
\usepackage{amsmath}
\usepackage[english,  activeacute]{babel}
\usepackage[latin1]{inputenc}
\usepackage{amssymb}
\usepackage{amsthm}
\usepackage{graphics,graphicx}
\usepackage{array}
\usepackage{cite}
\usepackage{a4wide}
\allowdisplaybreaks
\usepackage{color, url}
\usepackage{float}
\usepackage{caption}
\theoremstyle{plain}
\newtheorem{theorem}{Theorem}
\newtheorem{proposition}[theorem]{Proposition}
\newtheorem{corollary}[theorem]{Corollary}
\newtheorem{lemma}[theorem]{Lemma}
\theoremstyle{definition}

\newtheorem{conjecture}[theorem]{Conjecture}

\date{\today}

\title[On a conjecture of Lin and Kim]{On a conjecture of Lin and Kim concerning a refinement of Schr\"{o}der numbers}

\begin{document}

\author[T. Mansour]{Toufik Mansour}
\address{Department of Mathematics, University of Haifa,
3498838 Haifa, Israel}
\email{tmansour@univ.haifa.ac.il}
\author[M. Shattuck]{Mark Shattuck}
\address{Department of Mathematics, University of Tennessee,
37996 Knoxville, TN}
\email{shattuck@math.utk.edu}

\begin{abstract}
In this paper, we compute the distribution of the first letter statistic on nine avoidance classes of permutations corresponding to two pairs of patterns of length four.  In particular, we show that the distribution is the same for each class and is given by the entries of a new Schr\"{o}der number triangle.  This answers in the affirmative a recent conjecture of Lin and Kim.  We employ a variety of techniques to prove our results, including generating trees, direct bijections and the kernel method.  For the latter, we make use of in a creative way what we are trying to show in three cases to aid in solving a system of functional equations satisfied by the associated generating functions.
\end{abstract}

\maketitle

\section{Introduction}
Given two permutations $\pi=\pi_1\cdots\pi_n\in \mathcal{S}_n$ and $\tau=\tau_1\cdots\tau_k\in \mathcal{S}_k$, we say that $\pi$ \emph{contains} the pattern $\tau$ if there exist indices $i_1 < i_2 <\cdots< i_k$ such that
$\pi_{i_1}\cdots\pi_{i_k}$ is order isomorphic to $\tau$, that is, $\pi_{i_a}>\pi_{i_b}$ if and only if $\tau_a>\tau_b$. Otherwise, $\pi$ is said to \emph{avoid} the pattern $\tau$. Moreover, we say that $\pi$ avoids a set $L$ of patterns if it avoids each pattern in $L$, and let $\mathcal{S}_n(L)$ denote the subset of $\mathcal{S}_n$ whose members avoid $L$. In recent decades, the study of pattern avoidance in permutations has been the object of considerable attention (see, e.g., \cite{Kit} and references contained therein).

An \emph{inversion} within a permutation $\sigma=\sigma_1\cdots\sigma_n\in \mathcal{S}_n$ is an
ordered pair $(a,b)$ such that $1\leq a<b\leq n$ and $\sigma_a>\sigma_b$. The inversion sequence of $\sigma$ is given by $a_1\cdots a_n$, where $a_i$ records the number of entries of $\sigma$ to the right of $i$ and less than $i$ for each $i \in [n]$. The systematic study of patterns in inversion sequences
was initiated only recently in \cite{CMSW,MaSh}.

The (large) Schr\"oder number $S_n$ (see \cite[A006318]{Slo}) is defined recursively by
$$nS_n=3(2n-3)S_{n-1}-(n-3)S_{n-2}, \qquad n \geq 3,$$
with $S_1=1$ and $S_2=2$, and arises as the enumerator of several avoidance classes of permutations corresponding to a pair of patterns of length four.  In particular, combining the results from \cite{Gire,Kr,West}, one has the $S_n$ enumerates $\mathcal{S}_n(\sigma,\tau)$ for the following ten inequivalent pairs $(\sigma,\tau)$:
 \begin{align*}
 \text{I.}~(1234, 2134) &&  \text{II.}~(1324, 2314) && \text{III.}~(1342, 2341) && \text{IV.}~(3124, 3214) && \text{V.}~(3142, 3214)~\\
  \text{VI.}~(3412, 3421) && \text{VII.}~(1324, 2134) && \text{VIII.}~(3124, 2314)&& \text{IX.}~(2134, 3124) && \text{X.}~(2413, 3142).
\end{align*}
This answered in the affirmative a conjecture originally posed by Stanley (see \cite{Kr} for details).   Moreover, outside of symmetry, there are no other such pairs $(\sigma,\tau)$ for which  $|\mathcal{S}_n(\sigma,\tau)|=S_n$.  In this paper, we obtain a refinement of this result in several cases by considering distributions of certain statistics on the various classes.  For other refinements of the Schr\"oder numbers, see, e.g., \cite{BSS,PSu,RSi,Rog}.

\begin{table}[htp]
\begin{tabular}{l||llllll}
  $n\backslash k$ & 1 & 2 & 3 & 4 & 5 & 6\\\hline\hline
    1   & 1 &   &   &   &    &\\
    2   & 1 & 1 &   &   &    &\\
    3   & 2 & 2 & 2 &   &    &\\
    4   & 4 & 6 & 6 & 6 &    &\\
    5   & 8 & 16& 22& 22& 22 &\\
    6   &16 &40 &68 &90 & 90 & 90
\end{tabular}
\caption{The new Schr\"oder triangle $S_{n,k}$ for $1\leq k\leq n\leq 6$.}\label{tabsch}
\end{table}

In a recent paper, Lin and Kim \cite{LK} introduced a new triangle $S_{n,k}$ for Schr\"oder numbers in their study of inversion sequences; see Table \ref{tabsch} above. Here, we find it here more convenient for what follows to start from $k=1$ instead of $k=0$, as was done in \cite{LK}. Note that $S_{n,k}$ is given recursively by
  $$S_{n,k}=S_{n,k-1}+2S_{n-1,k}-S_{n-1,k-1}, \qquad 1 \leq k \leq n-2,$$
  with $S_{n,n}=S_{n,n-1}=S_{n,n-2}$ for $n\geq 3$ and $S_{1,1}=S_{2,1}=S_{2,2}=1$.
   Lin and Kim showed that $S_{n,k}$ enumerates the inversion sequences $\pi_1\cdots\pi_n\in\{I_n(021)\mid \pi_n=k \mod n\}$.  Then they state the following conjecture which provides a connection between inversion sequences and pattern avoidance in permutations.
\begin{conjecture}(Lin and Kim \cite{LK}.)
Let $(\nu,\mu)$ be a pair of patterns of length four. Then
$$S_{n,k} = |\{\sigma_1\cdots\sigma_n\in \mathcal{S}_n(\nu,\mu)\mid \sigma_n=k \}|$$
for all $1\leq k\leq n$ if and only if $(\nu,\mu)$ is one of the following nine pairs:
\begin{align*}
&(4321,3421),\, (3241,2341),\, (2431,2341),\, (4231,3241),\, (4231,2431),\\ &(4231,3421),\, (2431,3241),\, (3421,2431),\, (3421,3241).
\end{align*}
\end{conjecture}

Given $1 \leq i \leq n$, let $\mathcal{S}_{n,i}(\sigma,\tau)$ denote the set of permutations of length $n$ avoiding $\sigma$ and $\tau$ and starting with $i$.  Note that one may consider equivalently the distribution of the first letter statistic on the set of permutations avoiding the reversal of the two patterns in question in each case. Here, we confirm the conjecture by showing that $|\mathcal{S}_{n,i}(\sigma,\tau)|=S_{n,i}$ for each of the nine pairs above (where the patterns in each pair are reversed).  It is seen that these nine cases are derived from only six of the ten symmetry classes (I)-(X) above.  It should be remarked that the reversal, complement and inverse operations do not respect the first letter statistic and thus members of the same symmetry class do not have the same first letter distribution in general.  Moreover, these are the only nine pairs such that $|\mathcal{S}_{n,i}(\sigma,\tau)|=S_{n,i}$ for all $i$; see Table \ref{Tableofvalues} at the end of the introduction which rules out all other possible pairs $(\sigma,\tau)$ for which $|\mathcal{S}_n(\sigma,\tau)|=S_n$.  Furthermore, we have also computed the joint distribution with the descents statistic in six cases which confirms an additional conjecture of Lin and Kim \cite{LK}.

In several cases, we will make use of a generating function approach to establish the result.
Note that by the kernel method \cite{HouM}, one can show
\begin{equation}\label{introe1}
\sum_{n\geq1}\left(\sum_{k=1}^nS_{n,k}y^k\right)x^n=\frac{xy(2-3x-3y+3xy)+xy(x+y-xy)(xy+\sqrt{1-6xy+x^2y^2})}{2(1-2x-y+xy)},
\end{equation}
which reduces when $y=1$ to the well-known formula Schr\"{o}der number generating function given by
\begin{equation}\label{introe2}
\sum_{n\geq1}S_nx^n=\frac{1-x-\sqrt{1-6x+x^2}}{2}.
\end{equation}

{\footnotesize\begin{table}[htp]
\begin{tabular}{|l|l||l|l|}\hline
  $\sigma,\tau$ & $\{|\{\pi\in S_8(\sigma,\tau)\mid\pi_1=k\}|\}_{k=1}^8$ & $\sigma,\tau$ & $\{|\{\pi\in S_8(\sigma,\tau)\mid\pi_1=k\}|\}_{k=1}^8$ \\\hline\hline
2413,4123&	1584,1036,996,956,879,751,533,233    &
3142,4123&	1584,1584,1252,912,637,443,323,233   \\\hline
3142,3214&	1584,1584,736,396,292,304,488,1584   &
2341,2413&	1584,488,304,292,396,736,1584,1584   \\\hline
2341,3142&	1584,811,587,489,481,577,855,1584    &
2413,3214&	1584,855,577,481,489,587,811,1584    \\\hline
2431,3421&	1806,1022,710,614,644,795,1161,1806  &
2431,4231&	1806,1092,1008,1045,1120,1134,924,429\\\hline
2314,3124&	1806,1092,752,629,629,752,1092,1806  &
2314,3214&	1806,1092,752,629,629,752,1092,1806  \\\hline
2341,3241&	1806,1092,752,629,629,752,1092,1806  &
2413,3142&	1806,1092,752,629,629,752,1092,1806  \\\hline
2431,3241&	1806,1092,752,629,629,752,1092,1806  &
2134,3124&	1806,1161,795,644,614,710,1022,1806  \\\hline
3241,3421&	1806,1806,1198,678,406,342,516,1806  &
3214,3241&	1806,1806,1220,672,390,342,516,1806  \\\hline
3124,4123&	1806,1806,1502,1152,840,594,429,429  &
3214,4213&	1806,1806,1502,1152,840,594,429,429  \\\hline
3241,4231&	1806,1806,1502,1152,840,594,429,429  &
3412,4312&	1806,1806,1502,1152,840,594,429,429  \\\hline
3421,4231&	1806,1806,1502,1152,840,594,429,429  &
3421,4321&	1806,1806,1502,1152,840,594,429,429  \\\hline
{\it4123,4132}&	1806,1806,1806,1412,928,512,224,64   &
{\it4123,4213}&	1806,1806,1806,1412,928,512,224,64   \\\hline
{\it4132,4213}&	1806,1806,1806,1412,928,512,224,64   &
{\it4132,4231}&	1806,1806,1806,1412,928,512,224,64   \\\hline
{\it4132,4312}&	1806,1806,1806,1412,928,512,224,64   &
{\it4213,4231}&	1806,1806,1806,1412,928,512,224,64   \\\hline
{\it4213,4312}&	1806,1806,1806,1412,928,512,224,64   &
{\it4231,4312}&	1806,1806,1806,1412,928,512,224,64   \\\hline
{\it4312,4321}&	1806,1806,1806,1412,928,512,224,64   &
3124,3214&	1806,1806,788,540,484,540,788,1806   \\\hline
3412,3421&	1806,1806,788,540,484,540,788,1806   &
2314,2341&	1806,516,342,390,672,1220,1806,1806  \\\hline
2134,2314&	1806,516,342,406,678,1198,1806,1806  &
2134,2143&	1806,788,540,484,540,788,1806,1806   \\\hline
2341,2431&	1806,788,540,484,540,788,1806,1806   &
1432,2413&	233,323,443,637,912,1252,1584,1584   \\\hline
1432,3142&	233,533,751,879,956,996,1036,1584    &
1234,2134&	429,429,594,840,1152,1502,1806,1806  \\\hline
1243,2143&	429,429,594,840,1152,1502,1806,1806  &
1324,2134&	429,429,594,840,1152,1502,1806,1806  \\\hline
1324,2314&	429,429,594,840,1152,1502,1806,1806  &
1342,2341&	429,429,594,840,1152,1502,1806,1806  \\\hline
1432,2431&	429,429,594,840,1152,1502,1806,1806  &
1324,3124&	429,924,1134,1120,1045,1008,1092,1806\\\hline
1423,4123&	429,924,1344,1582,1582,1344,924,429  &
1432,4132&	429,924,1344,1582,1582,1344,924,429  \\\hline
{\bf1234,1243}&	64,224,512,928,1412,1806,1806,1806   &
{\bf1243,1324}&	64,224,512,928,1412,1806,1806,1806   \\\hline
{\bf1243,1342}&	64,224,512,928,1412,1806,1806,1806   &
{\bf1243,1423}&	64,224,512,928,1412,1806,1806,1806   \\\hline
{\bf1324,1342}&	64,224,512,928,1412,1806,1806,1806   &
{\bf1324,1423}&	64,224,512,928,1412,1806,1806,1806   \\\hline
{\bf1342,1423}&	64,224,512,928,1412,1806,1806,1806   &
{\bf1342,1432}&	64,224,512,928,1412,1806,1806,1806   \\\hline
{\bf1423,1432}&	64,224,512,928,1412,1806,1806,1806   &&\\\hline
\end{tabular}
\caption{All symmetries of classes I-X according to first letter statistic.  Note the 9 boldface cases corresponding to the conjecture and the 9 italicized cases obtained by complementation.}\label{Tableofvalues}
\end{table}}

This paper is organized as follows.  In the next section, we show six cases of the conjecture above using various methods such as induction, bijections and generating trees.  In the third section, we prove the remaining three cases, each of which involves $1243$ and another pattern, by considering the joint distribution of the first and second letter statistics.  This allows one to write a system of recurrence relations in each case which may then be expressed in terms of some auxiliary generating functions leading to a system of functional equations.  At this point, one can use the conjecture itself in these particular cases along with the kernel method to ascertain a potential solution to the aforementioned system, which may then be shown to be the actual solution.  Taking the variable that marks the second letter statistic to be unity then recovers formula \eqref{introe1} and demonstrates the desired equality of distributions.

In several instances, the distribution of the first letter statistic on the pattern pair in question follows as a special case of a more general distribution.  For one's reference, listed below are the places within the paper where the specific cases are shown.
\begin{table}[htp]
\begin{tabular}{|l|l||l|l||l|l|} \hline
  Pattern Pair & Reference & Pattern Pair & Reference & Pattern Pair & Reference \\\hline\hline
  1234,1243  &  \text{Theorem}~\ref{Th1}&
  1243,1324  &  \text{Corollary}~\ref{co1243x1324}&
  1243,1342  &  \text{Corollary}~\ref{co1243x1342}\\\hline
  1243,1423   & \text{Corollary}~\ref{co1243x1423}&
  1324,1342   &  \text{Theorem}~\ref{Th1}&
  1324,1423   &  \text{Theorem}~\ref{gtreTh}\\\hline
  1342,1423   &   \text{Theorem}~\ref{gtreTh}&
  1342,1432   &   \text{Theorem}~\ref{1342,1432}&
 1423,1432   &   \text{Theorem}~\ref{Th1}\\\hline
\end{tabular}
\caption{Places where specific cases of Conjecture 1 are proven.}\label{maintable}
\end{table}

\section{Pattern avoidance and the new Schr\"{o}der triangle}

In this section, we enumerate members of $\mathcal{S}_{n,i}(\sigma,\tau)$ confirming the conjecture in six of the cases, the first three of which we treat together in the following result.

\begin{theorem}\label{Th1}
If $n \geq 1$ and $1 \leq i \leq n$, then $|\mathcal{S}_{n,i}(\sigma,\tau)|=S_{n,i}$ for $(\sigma,\tau)=(1234,1243)$, $(1324,1342)$ and $(1423,1432)$.
\end{theorem}
\begin{proof}
We prove the case when $(\sigma,\tau)=(1234,1243)$, the others being similar.  Let $\mathcal{A}_{n,i}=\mathcal{S}_{n,i}(1234,1243)$ and we first write a recurrence for $a_{n,i}=|\mathcal{A}_{n,i}|$.  Note that members of $\mathcal{A}_{n,i}$ where $1 \leq i \leq n-3$ must have second letter $\ell \in [i-1]$ or $\ell=n-1,n$, for otherwise there would be an occurrence of $1234$ or $1243$ starting with $i\ell$.  Furthermore, the first letter $i$ is seen to be extraneous concerning the avoidance of $1234$ or $1243$ if $\ell \in [i-1]$ and thus may be deleted.  Similarly, $\ell$ may deleted in cases when $\ell=n-1$ or $n$.  This implies the recurrence
$$a_{n,i}=2a_{n-1,i}+\sum_{\ell=1}^{i-1}a_{n-1,\ell}, \qquad 1 \leq i \leq n-2,$$
with $a_{n,n}=a_{n,n-1}=\sum_{i=1}^{n-1}a_{n-1,i}$ for $n \geq 3$.

Before proceeding further, note that
\begin{align*}
S_{n,n-2}&=\sum_{k=1}^{n-2}(S_{n,k}-S_{n,k-1})=\sum_{k=1}^{n-2}(2S_{n-1,k}-S_{n-1,k-1})=2S_{n-1,n-2}+\sum_{k=1}^{n-3}S_{n-1,k}\\
&=\sum_{k=1}^{n-1}S_{n-1,k}.
\end{align*}
Thus, the $a_{n,i}$ and $S_{n,i}$ are both given by the sum of the entries of the previous row if $i\in[n-2,n]$,  with the two arrays also agreeing for $n=1,2$.  Therefore, to complete the proof, it suffices to show
\begin{equation}\label{Th1e1}
S_{n,i}=2S_{n-1,i}+\sum_{\ell=1}^{i-1}S_{n-1,\ell}, \qquad 1 \leq i \leq n-2.
\end{equation}
To do so, we proceed by induction on $n$ and $i$, the $i=1$ case clearly holding since $S_{n,1}=2S_{n-1,1}$ for all $n \geq 3$.  So assume $n \geq 4$ and $2 \leq i \leq n-2$.  Note further that \eqref{Th1e1} also holds when $i=n-1$ for $n \geq 3$, which follows from the work above.
Then by the induction hypothesis, we have
\begin{align}
S_{n,i}&=S_{n,i-1}+2S_{n-1,i}-S_{n-1,i-1}=2S_{n-1,i-1}+\sum_{\ell=1}^{i-2}S_{n-1,\ell}+4S_{n-2,i}+2\sum_{\ell=1}^{i-1}S_{n-2,\ell}\notag\\
&\quad-2S_{n-2,i-1}-\sum_{\ell=1}^{i-2}S_{n-2,\ell}\notag\\
&=2S_{n-1,i-1}+4S_{n-2,i}+\sum_{\ell=1}^{i-2}(S_{n-1,\ell}+S_{n-2,\ell}).\label{Th1e2}
\end{align}
Upon substituting \eqref{Th1e1} into \eqref{Th1e2}, and simplifying the resulting equality, to complete the induction, we must show
\begin{equation}\label{Th1e3}
2S_{n-1,i}=S_{n-1,i-1}+4S_{n-2,i}+\sum_{\ell=1}^{i-2}S_{n-2,\ell}.
\end{equation}
Substituting $S_{n-1,i}=2S_{n-2,i}+\sum_{\ell=1}^{i-1}S_{n-2,\ell}$ into \eqref{Th1e3} reduces it to
$$2\sum_{\ell=1}^{i-1}S_{n-2,\ell}=S_{n-1,i-1}+\sum_{\ell=1}^{i-2}S_{n-2,\ell},$$
i.e.,
$$S_{n-1,i-1}=2S_{n-2,i-1}+\sum_{\ell=1}^{i-2}S_{n-2,\ell},$$
which is true by the induction hypothesis.  This establishes \eqref{Th1e1} and completes the proof in the case of $(\sigma,\tau)=(1234,1243)$.  Since the associated $a_{n,i}$ can be shown to satisfy the same recurrence and initial conditions for the other two pattern pairs, one obtains the same result in these cases as well.
\end{proof}

To establish the case $(1342,1432)$, we define a bijection with a previous case.

\begin{theorem}\label{1342,1432}
The members of $\mathcal{S}_{n,i}(1342,1432)$ having a prescribed set of left-right minima in specified positions are in one-to-one correspondence with members of $\mathcal{S}_{n,i}(1234,1243)$ having the same set of left-right minima in the same positions.  In particular, $|\mathcal{S}_{n,i}(1342,1432)|=S_{n,i}$.
\end{theorem}
\begin{proof}
We define a bijection $f$ between $\mathcal{S}_{n,i}(1342,1432)$ and $\mathcal{S}_{n,i}(1234,1243)$ with the desired properties as follows.  Let $\pi=\pi_1\cdots\pi_n \in \mathcal{S}_{n,i}(1342,1432)$ have left-right minima values $a_r>a_{r-1}>\cdots>a_1$, where $a_r=\pi_1$ and $a_1=1$.  Note that since the patterns in both pairs start with $1$, one may always assume a left-right minima plays the role of the $1$.  Suppose $\pi=\alpha 1 \beta$, where $\alpha$ or $\beta$ is possibly empty.  Then let $\pi_1$ be obtained by reversing the order of all letters in $\beta$ within $\pi_0=\pi$, that is, $\pi_1=\alpha 1 \beta'$, where $\beta'$ denotes the reversal of $\beta$.  Note that $\pi_1$ contains no $1234$ or $1243$ starting with the actual element $1$ since all such occurrences of $1234$ or $1243$ within $\pi_0$ have been replaced with comparable occurrences of $1432$ or $1342$, respectively.

If $r=1$ (i.e., if $\pi$ starts with $1$), then we let $f(\pi)=\pi_1$ and we are done, so assume $r \geq 2$.  In this case, let $S$ denote the subsequence comprising all elements of $[a_2+1,n]$ occurring to the right of $a_2$ in the permutation $\pi_1$.  Let $S^*$ denote the portion of $S$ to the right of $a_1$ within $\pi_1$.  We then let $\pi_2$ be the permutation obtained from $\pi_1$ as follows.  First remove all letters of $\pi_1$ corresponding to $S$ and replace them with blanks.  Within these blanks, from left to right, we then write the elements of $S^*$ followed by the reversal of $S-S^*$ to obtain $\pi_2$.  That is, the elements between $a_2$ and $a_1$ in $\pi_1$ that belong to $[a_2+1,n]$ have their relative order reversed when they are written within the blanks and now follow (instead of precede) the remaining elements of $S$.  Observe that $\pi_2$ has no occurrences of $1234$ or $1243$ starting with $a_2$ as the relative order of all elements belonging to $[a_2+1,n]$ to the right of $a_2$ in $\pi_2$ is the reverse of the order of these same elements in $\pi$.  To see this, note that only the elements in $S-S^*$ have their order reversed in going from $\pi_1$ to $\pi_2$ as the order of $S^*$ in $\pi_1$ is already the reversal of what it was in $\pi$.  Further, no occurrence of $1234$ or $1243$ starting with $a_1$ can arise during the transition from $\pi_1$ to $\pi_2$ since the positions of elements in $[2,a_2-1]$ do not change during this transition.

If $r=2$, then let $f(\pi)=\pi_2$.  Otherwise, consider moving the letters that belong to $[a_3+1,n]$ and lie to the right of $a_3$ within $\pi_2$ in a comparable manner as before, reversing the order of only those elements occurring between $a_3$ and $a_2$.  Continue on for subsequently larger $i$ where in the $r$-th step, one moves letters in $[a_r+1,n]$ in obtaining $\pi_{r}$ from $\pi_{r-1}$.  Let $f(\pi)=\pi_r$.  Note that $\pi_r$ indeed belongs of $\mathcal{S}_n(1234,1243)$ since as one may verify no occurrence of $1234$ or $1243$ starting with $a_j$ for some $j<i$ can arise during the $i$-th transition from $\pi_{i-1}$ to $\pi_i$ for any $i \in [r]$.  Since each step of the algorithm described above is seen to preserve both the positions and values of left-right minima, then so does the mapping $f$ (in particular, the first letter statistic is preserved by $f$).  This implies that the inverse  of $f$ may be found by reversing each of the $r$ steps of the algorithm starting with the last step and proceeding in reverse order.
\end{proof}

We next treat the cases $(1324,1423)$ and $(1342,1423)$ together using a generating tree approach (see, e.g., \cite{West2}).  Consider forming $\pi \in \mathcal{S}_{n}(1324,1423)$ (or $\mathcal{S}_{n}(1342,1423)$) by inserting the element $1$ within a member $\rho \in \mathcal{S}_{n-1}(1324,1423)$ ($\mathcal{S}_{n-1}(1342,1423)$, respectively), expressed using the letters in $[2,n]$.  By an \emph{active site} within $\pi=\pi_1\cdots\pi_{n-1}$, we mean a position in which $1$ may be inserted without introducing an occurrence of either $1324$ or $1423$ (and likewise for $(1342,1423)$).  Given either pair $(\sigma,\tau)$ of patterns under consideration, let $u_n(i,j)$ denote the number of members of $\mathcal{S}_{n,i}(\sigma,\tau)$ having exactly $j$ active sites.  Note that $u_n(i,j)$ for $n \geq 2$ can assume non-zero values only when $1 \leq i \leq n$ and $3 \leq j \leq n+1$.

The $u_n(i,j)$ are given recursively as follows.

\begin{lemma}\label{gtreL1}
If $3 \leq j \leq n$, then
\begin{equation}\label{gtreL1e1}
u_n(i,j)=u_{n-1}(i-1,j-1)+\sum_{\ell=j-1}^nu_{n-1}(i-1,\ell), \qquad 2 \leq i \leq n,
\end{equation}
with $u_n(1,j)=0$.  If $j=n+1$, then we have
\begin{equation}\label{gtreL1e2}
u_n(i,n+1)=u_{n-1}(i-1,n), \qquad i \geq 2,
\end{equation}
with $u_n(1,n+1)=2^{n-2}$ for $n \geq 2$ and $u_1(1,2)=1$.
\end{lemma}
\begin{proof}
We treat the case $(1324,1423)$, with the same recurrences seen to hold for $(1342,1423)$ by a comparable analysis.  Let $\mathcal{U}_{n,i,j}$ denote the subset of $\mathcal{S}_{n,i}(1324,1423)$ enumerated by $u_n(i,j)$.  Note that the (active) sites of $\pi \in \mathcal{U}_{n,i,j}$ correspond to the rightmost $j$ possible positions of $\pi$ in which to insert a $1$ if $j \leq n$.  For if the $j$-th letter $x$ from the right within $\pi$ where $j \leq n$ starts either a $213$ or $312$ and is the rightmost such letter to do so, then all positions to the right of $x$ are sites.  From this observation, we may conclude that $\mathcal{U}_{n,1,j}$ is empty if $3 \leq j \leq n$, upon considering separately the cases $j=n$ or $j<n$.  On the other hand, since $1$ starts both of the patterns that are being avoided, we have that $\mathcal{U}_{n,1,n+1}$ is synonymous with $\mathcal{S}_{n-1}(213,312)$, which has cardinality $2^{n-2}$ if $n \geq 2$.  This establishes the initial conditions when $i=1$, so assume henceforth that $i \geq 2$.

Let $\pi$ be formed from a precursor $\alpha \in \mathcal{U}_{n-1,a,b}$ for some $a$ and $b$, expressed using $[2,n]$, by inserting $1$ as described.  Let $(k)$ denote a precursor having $k$ sites.  Then we have the succession rule $(k)\rightarrow(3)(4)\cdots(k+1)(k+1)$ with root $(2)$, which follows from the argument used in the proof of \cite[Prop.~11]{Kr} or can be reasoned directly in this case.  Note that no offspring of $\alpha$ if $b<n$ can arise by inserting $1$ in the first position, for otherwise a $1324$ or $1423$ would be introduced, whereas if $b=n$, an offspring so produced would have $n+1$ sites.  Since any offspring $\pi \in \mathcal{U}_{n,i,j}$ with $i \geq 2$ must have precursor starting with $i-1$ and at least $j-1$ sites, recurrence \eqref{gtreL1e1} follows.  On the other hand, an offspring $\pi \in \mathcal{U}_{n,i,n+1}$ where $i \geq 2$ can only come about by inserting $1$ in the final position within its precursor $\rho \in \mathcal{U}_{n-1,i-1,n}$, for the other offspring of $\rho$ having $n+1$ sites comes about by inserting $1$ in the first position.  This implies \eqref{gtreL1e2} and completes the proof.
\end{proof}

Given $n \geq 2$, let $v_n(i;q)=\sum_{j=3}^{n+1}u_n(i,j)q^{j-2}$ for $1 \leq i \leq n$, with $v_1(1;q)=1$.  Define the joint distribution polynomial $v_n(y,q)=\sum_{i=1}^nv_n(i;q)y^i$ for $n \geq 2$, with $v_1(y,q)=y$.

Then the $v_n(y,q)$ are given recursively as follows.

\begin{lemma}\label{gtreL2}
If $n \geq 3$, then
\begin{equation}\label{gtreL2e1}
v_n(y,q)=\frac{(1-y)(2^{n-1}y-y^n)q^{n-1}}{2-y}+\frac{yq(v_{n-1}(y,1)+(1-2q)v_{n-1}(y,q))}{1-q}, \qquad n \geq 2,
\end{equation}
with $v_1(y,q)=y$.
\end{lemma}
\begin{proof}
Note that \eqref{gtreL2e1} is seen to hold if $n=2$ since $v_2(y,q)=yq(1+y)$, so assume $n \geq 3$.  Multiplying both sides of \eqref{gtreL1e1} by $q^{j-2}$, summing over $3 \leq j \leq n$ and adding $q^{n-1}$ times \eqref{gtreL1e2} gives
\begin{align*}
v_n(i;q)&=u_{n-1}(i-1,n)q^{n-1}+\sum_{j=3}^nu_{n-1}(i-1,j-1)q^{j-2}+\sum_{j=3}^nq^{j-2}\sum_{\ell=j-1}^nu_{n-1}(i-1,\ell)\\
&=q\sum_{j=3}^nu_{n-1}(i-1,j)q^{j-2}+\sum_{\ell=2}^nu_{n-1}(i-1,\ell)\sum_{j=3}^{\ell+1}q^{j-2}-u_{n-1}(i-1,n)q^{n-1}\\
&=qv_{n-1}(i-1;q)+\sum_{\ell=2}^nu_{n-1}(i-1,\ell)\frac{q-q^\ell}{1-q}-u_{n-1}(i-1,n)q^{n-1}\\
&=qv_{n-1}(i-1;q)+\frac{q}{1-q}(v_{n-1}(i-1;1)-qv_{n-1}(i-1;q))-u_{n-1}(i-1,n)q^{n-1},
\end{align*}
for $2 \leq i \leq n$, with $v_n(1;q)=2^{n-2}q^{n-1}$.  Multiplying the last recurrence by $y^i$, and summing over $i$, yields
\begin{align*}
v_n(y,q)&=yqv_{n-1}(y,q)+\frac{yq(v_{n-1}(y,1)-qv_{n-1}(y,q))}{1-q}+2^{n-2}yq^{n-1}-q^{n-1}\sum_{i=2}^nu_{n-1}(i-1,n)y^i\\
&=\frac{yq(v_{n-1}(y,1)+(1-2q)v_{n-1}(y,q))}{1-q}+2^{n-2}yq^{n-1}-q^{n-1}\left(y^n+\sum_{i=2}^{n-1}2^{n-i-1}y^i\right),
\end{align*}
where we have used the fact
\begin{align*}
u_m(j,m+1)&=|\mathcal{S}_{m,j}(213,312)|=\begin{cases}
	2^{m-j-1}, &\text{if } 1 \leq j <m;\\[3 pt]
	1, &\text{if } j=m,
	\end{cases}
\end{align*}
which follows from elementary considerations.  Simplifying and combining the inhomogeneous terms in the last recurrence formula now gives \eqref{gtreL2e1}.
\end{proof}

Note that \eqref{gtreL2e1} holds for both $(1324,1423)$ and $(1342,1423)$ since Lemma \ref{gtreL1} applies to either pair.   Let $f(x,y;q)=\sum_{n\geq1}v_n(y,q)x^n$.

\begin{theorem}\label{gtreTh}
The generating function for the joint distribution of the first letter and number of active sites statistics on $\mathcal{S}_n(1324,1423)$ and $\mathcal{S}_n(1342,1423)$ is given by
\begin{equation}\label{gtreThe1}
f(x,y;q)=\frac{xy(1-q)(1-xq)(1-2xyq)}{(1-2xq)(1-xyq)(1-(xy+1)q+2xyq^2)}+\frac{xyq}{1-(xy+1)q+2xyq^2}A(x,y),
\end{equation}
where
$$A(x,y)=\frac{xy(2-3x-3y+3xy)+xy(x+y-xy)(xy+\sqrt{1-6xy+x^2y^2})}{2(1-2x-y+xy)}.$$
\end{theorem}
\begin{proof}
Multiplying both sides of \eqref{gtreL2e1} by $x^n$, and summing over $n \geq 2$, yields
\begin{align*}
f(x,y;q)-xy&=\frac{y(1-y)}{2-y}\cdot\frac{2x^2q}{1-2xq}-\frac{1-y}{2-y}\cdot\frac{x^2y^2q}{1-xyq}+\frac{xyq}{1-q}f(x,y;1)\\
&\quad+\frac{xyq(1-2q)}{1-q}f(x,y;q),
\end{align*}
which may be rewritten as
\begin{align}\label{gtreThe2}
\left(1-\frac{xyq(1-2q)}{1-q}\right)f(x,y;q)=xy+\frac{x^2yq(1-y)}{(1-2xq)(1-xyq)}+\frac{xyq}{1-q}f(x,y;1).
\end{align}
To solve \eqref{gtreThe2}, we apply the kernel method \cite{HouM} and let
$$q=q_0=\frac{1+xy-\sqrt{1-6xy+x^2y^2}}{4xy}.$$
This gives
\begin{align*}
f(x,y;1)&=\frac{q_0-1}{q_0}+\frac{x(1-y)(q_0-1)}{1-2xq_0-xyq_0+2x^2yq_0^2}=\frac{q_0-1}{q_0}+\frac{x(1-y)(q_0-1)}{1-x+x(xy-y-1)q_0},
\end{align*}
where we have used the fact $2xyq_0^2=(1+xy)q_0-1$.  Substituting $\frac{1}{q_0}=\frac{1+xy+\sqrt{1-6xy+x^2y^2}}{2}$ into
$$f(x,y;1)=\left(1-\frac{1}{q_0}\right)\left(1+\frac{x(1-y)}{\frac{1-x}{q_0}+x(xy-y-1)}\right),$$
and simplifying, leads to the formula $f(x,y;1)=A(x,y)$.  Substituting this back into \eqref{gtreThe2} yields \eqref{gtreThe1}.
\end{proof}

\section{The remaining cases}

In this section, we consider the remaining cases $(1243,1423)$, $(1243,1342)$ and $(1423,1324)$.  We adopt a common approach for these three cases and consider the joint distribution of the first and second letter statistics.  For the given pattern pair $(\sigma,\tau)$ under current consideration, let $a_n(i,j)$ denote the number of members of $\mathcal{S}_{n,i}(\sigma,\tau)$ whose second letter is $j$.  Let $a_n(i)=\sum_{j\neq i} a_n(i,j)$ and $a_n=\sum_{i=1}^na_{n,i}$.  Note that $a_n=S_n$, the $n$-th Schr\"{o}der number (see, e.g., \cite{Kr}).  Clearly, we have
\begin{equation}\label{12431423e0}
a_n(i,j)=a_{n-1}(j), \quad 1 \leq j <i,
\end{equation}
for all pairs of patterns under consideration.  If $i<j$ where $i \leq n-3$, then in order to write a recurrence for $a_n(i,j)$ in this case, we consider the position of the second ascent and further subcases based on the size of the difference $j-i$.  If $n \leq 3$, then all three pattern pairs have initial values given by $a_1(1)=a_2(1,2)=a_2(2,1)=a_3(i,j)=1$ where $i,j \in [3]$ with $i \neq j$.

\subsection{The case $(1243,1423)$.}

In this subsection, we enumerate the members of $\mathcal{S}_n(1243,1423)$ according to the joint distribution of the first and second letter statistics.

We first prove the following recurrence for $a_n(i,j)$ when $i<j$.

\begin{lemma}\label{12431423L1}
We have
\begin{equation}\label{12431423e1}
a_n(i,i+1)=a_{n-1}(i,i+1)+\sum_{a=1}^{i-1}\sum_{c=0}^{i-a-1}\sum_{b=a+1}^{i-c}\binom{i-a-1}{c}a_{n-c-2}(a,b), \qquad 1 \leq i \leq n-2,
\end{equation}
with $a_n(n-1,n)=a_{n-2}$,
\begin{align}
a_n(i,i+2)&=a_{n-1}(i,i+1)+a_{n-1}(i,i+2)\notag\\
&\quad+\sum_{a=1}^{i-1}\sum_{c=0}^{i-a-1}\sum_{b=a+1}^{i-c+1}\binom{i-a-1}{c}a_{n-c-2}(a,b), \qquad 1 \leq i \leq n-3,\label{12431423e2}
\end{align}
with $a_n(n-2,n)=a_{n-2}$, and
\begin{align}
a_n(i,j)&=a_{n-1}(i,j-1)+\sum_{a=1}^{i-1}\sum_{c=0}^{i-a-1}\binom{i-a-1}{c}a_{n-c-2}(a,j-c-2)\notag\\
&\quad+(1-\delta_{j,n})\cdot\left(a_{n-1}(i,j)+\sum_{a=1}^{i-1}\sum_{c=0}^{i-a-1}\binom{i-a-1}{c}a_{n-c-2}(a,j-c-1)\right) \label{12431423e3}
\end{align}
for $4 \leq i+3 \leq j \leq n$.
\end{lemma}
\begin{proof}
Let $k$ denote the third letter of a member of $\mathcal{A}_{n,i,j}$.  We will consider various cases based on $k$ in the proofs of \eqref{12431423e1}-\eqref{12431423e3}.  To show \eqref{12431423e1}, first observe that for members of $\mathcal{A}_{n,i,i+1}$ where $i \leq n-2$, one has either $k=i+2$ or $k<i$, for $k>i+2$ would ensure an occurrence of $1243$ which isn't permissible.  In the first case, the third letter is superfluous concerning avoidance of either pattern and thus may be deleted leading to $a_{n-1}(i,i+1)$ possibilities.  On the other hand, in the latter case, we have that $\pi \in \mathcal{A}_{n,i,i+1}$ can be decomposed as $\pi=i(i+1)x_1\cdots x_cad\rho$, where $a \in [i-1]$, $x_1,\ldots,x_c \in [a+1,i-1]$ with $x_1>\cdots>x_c$, $d>a$ and $\rho$ denotes the terminal section of $\pi$.  Note that $d=i+2$ or $d \in [a+1,i-1]-\{x_1,\ldots,x_c\}$ since $d>i+2$ is again not allowed.  Further, the letters $x_1,\ldots,x_c,i,i+1$ may be removed from $\pi$ in light of the ascent $a,d$, leading to a member $\pi'\in \mathcal{A}_{n-c-2,a,b}$ for some $b$ after reduction of letters.  Note that within $\pi'$, the $b$ parameter value can range over $[a+1,i-c]$ since the largest possible value it may assume is $i+2-(c+2)=i-c$.  As there are $\binom{i-a-1}{c}$ ways in which to choose the $x_i$, considering all possible $a$, $c$ and $b$ implies that the number of $\pi \in \mathcal{A}_{n,i,i+1}$ for which $k<i$ is given by $\sum_{a=1}^{i-1}\sum_{c=0}^{i-a-1}\sum_{b=a+1}^{i-c}\binom{i-a-1}{c}a_{n-c-2}(a,b)$, which establishes \eqref{12431423e1}.

A similar argument applies to \eqref{12431423e2} except that now we have $k<i$ or $k=i+1,i+3$ for members of $\mathcal{A}_{n,i,i+2}$ where $i \leq n-3$.  Note that $k=i+3$ leads to $a_{n-1}(i,i+2)$ possibilities, as $k$ may be deleted in this case.  If $k<i$, then making use of the previous notation, we have $d \in [a+1,i-1]\cup\{i+1,i+3\}$ and thus $b$ may range in $[a+1,i-c+1]$ in this case. Finally, to show \eqref{12431423e3}, we consider first the case when $j<n$ within a member of $\mathcal{A}_{n,i,j}$ where $j \geq i+3$.  Here, we would have $k\in[i-1]\cup\{j-1,j+1\}$, as $k\in[i+1,j-2]$ would lead to an occurrence of 1423 as witnessed by $ijk(k+1)$.  If $k=j-1$ or $k=j+1$, then $k$ may be deleted in either case giving $a_{n-1}(i,j-1)$ and $a_{n-1}(i,j)$ possibilities, respectively.  If $k<i$, then we must have $d=j-1$ or $d=j+1$, for other values of $d$ would lead to an occurrence of $1243$ or $1423$.  Considering all possible $a$ and $c$ then yields the two double sum expressions occurring in the $j<n$ case of \eqref{12431423e3}.  Combining this with the preceding then implies \eqref{12431423e3} when $j<n$. On the other hand, if $j=n$, then $k=j+1$ or $d=j+1$ does not occur in the preceding argument, which implies \eqref{12431423e3} in this case and completes the proof.
\end{proof}

We will make use of the following generating functions:
\begin{align*}
A(x,v,w)&=\sum_{n\geq2}\sum_{a=1}^n\sum_{b=1}^na_n(a,b)v^aw^bx^n,\\
A^+(x,v,w)&=\sum_{n\geq2}\sum_{a=1}^{n-1}\sum_{b=a+1}^na_n(a,b)v^aw^bx^n,\\
A^-(x,v,w)&=\sum_{n\geq2}\sum_{a=2}^n\sum_{b=1}^{a-1}a_n(a,b)v^aw^bx^n,\\
C(x,v)&=\sum_{n\geq2}\sum_{a=1}^{n-1}a_n(a,a+1)v^ax^n,\\
D(x,v)&=\sum_{n\geq2}\sum_{a=1}^{n-2}a_n(a,a+2)v^ax^n,\\
B(x,v,w)&=\sum_{n\geq4}\sum_{a=1}^{n-3}\sum_{b=a+3}^na_n(a,b)v^aw^bx^n.
\end{align*}
Clearly, $A^+(x,v,w)=wC(x,vw)+w^2D(x,vw)+B(x,v,w)$, by the definitions.

Translating \eqref{12431423e0}-\eqref{12431423e3} in terms of generating functions yields
\begin{align*}
A^-(x,v,w)&=\sum_{n\geq2}\sum_{a=2}^n\sum_{b=1}^{a-1}a_{n-1}(b)v^aw^bx^n\\
&=\sum_{n\geq2}\sum_{b=1}^{n-1}a_{n-1}(b)\frac{v^{b+1}-v^{n+1}}{1-v}w^bx^n\\
&=v^2wx^2+\frac{vx}{1-v}A(x,vw,1)-\frac{v^2x}{1-v}A(vx,w,1),
\end{align*}
\begin{align*}
&wC(x,vw)-vw^2x^2A(vwx,1,1)-vw^2x^2-v^2w^3x^3\\
&=wxC(x,vw)
+\sum_{n\geq2}\sum_{i=1}^{n-2}\sum_{a=1}^{i-1}\sum_{c=0}^{i-a-1}\sum_{b=a+1}^{i-c}\binom{i-a-1}{c}a_{n-c-2}(a,b)v^iw^{i+1}x^n\\
&=wxC(x,vw)
+w\sum_{a\geq1}\sum_{c\geq0}\sum_{i\geq1}\sum_{n\geq i+c+a+2}\sum_{b=a+1}^{i+a}\binom{i+c-1}{c}a_{n-c-2}(a,b)(vw)^{i+a+c}x^n\\
&=wxC(x,vw)+w\sum_{a\geq1}\sum_{i\geq1}\sum_{n\geq a+i}\sum_{b=a+1}^{i+a}a_{n}(a,b)\frac{(vw)^{i+a}x^{n+2}}{(1-vwx)^i}\\
&=wxC(x,vw)\\
&+\frac{w}{vwx+vw-1}\sum_{n\geq2}\sum_{a=1}^{n-1}\sum_{b=a+1}^na_n(a,b)(\frac{x^{n+2}(vw)^{n+1}}{(1-vwx)^{n-a}}-\frac{x^{n+2}(vw)^b}{(1-vwx)^{b-a-1}})\\
&=wxC(x,vw)\\
&+\frac{vw^2x^2}{vwx+vw-1}A^+(\frac{vwx}{1-vwx},1-vwx,1)-\frac{(1-vwx)wx^2}{vwx+vw-1}A^+(x,1-vwx,\frac{vw}{1-vwx}),
\end{align*}
\begin{align*}
&w^2D(x,vw)-w^2x^2A(vwx,1,1)+v^2w^4x^4\\
&=w^2x(C(x,vw)-vwx^2A(vwx,1,1))+w^2xD(x,vw)\\
&+w^2\sum_{n\geq3}\sum_{i=1}^{n-3}\sum_{a=1}^{i-1}\sum_{c=0}^{i-1-a}\sum_{b=a+1}^{i-c+1}
\binom{i-1-a}{c}a_{n-2-c}(a,b)(vw)^ix^n\\
&=w^2x(C(x,vw)-vwx^2A(vwx,1,1))+w^2xD(x,vw)\\
&+w^2\sum_{a\geq1}\sum_{c\geq0}\sum_{i\geq1}\sum_{n\geq i+a+3}\sum_{b=a+1}^{i+a+1}\binom{i+c-1}{c}a_{n-2}(a,b)(vw)^{i+a+c}x^{n+c}\\
&=w^2x(C(x,vw)-vwx^2A(vwx,1,1))+w^2xD(x,vw)\\
&+w^2x^2\sum_{a\geq1}\sum_{n\geq a+2}\sum_{b=a+2}^na_n(a,b)\frac{(vw)^{i+a}x^n}{(1-vwx)^i}
+w^2x^2\sum_{a\geq1}\sum_{n\geq a+2}\sum_{i=1}^{n-a-1}a_n(a,a+1)\frac{(vw)^{i+a}x^n}{(1-vwx)^i}\\
&=w^2x(C(x,vw)-vwx^2A(vwx,1,1))+w^2xD(x,vw)\\ &+\frac{w^2x^2(1-vwx)}{vwx+vw-1}A^+(\frac{vwx}{1-vwx},1-vwx,1)
-\frac{wx^2(1-vwx)^2}{v(vwx+vw-1)}A^+(x,1-vwx,\frac{vw}{1-vwx})\\
&-w^2x^2C(x,vw),
\end{align*}
and
\begin{align*}
&B(x,v,w)\\
&=wxB(x,v,w)+w^3xD(x,vw)\\
&+\sum_{n\geq3}\sum_{i=1}^{n-3}\sum_{j=i+3}^n\sum_{a=1}^{i-1}\sum_{c=0}^{i-a-1}
\binom{i-a-1}{c}a_{n-c-2}(a,j-c-2)v^iw^jx^n\\
&+\sum_{n\geq3}\sum_{i=1}^{n-3}\sum_{j=i+3}^n\biggl((1-\delta_{j,n})(a_{n-1}(i,j)v^iw^jx^n\\
&\qquad\qquad\qquad\qquad\qquad\qquad+\sum_{a=1}^{i-1}\sum_{c=0}^{i-a-1}\binom{i-a-1}{c}a_{n-c-2}(a,j-c-1)v^iw^jx^n\biggr)\\
&=wxB(x,v,w)+w^3xD(x,vw)\\
&+\sum_{n\geq3}\sum_{i=1}^{n-3}\sum_{j=i+3}^n\sum_{a=1}^{i-1}\sum_{c=0}^{i-a-1}
\binom{i-a-1}{c}a_{n-c-2}(a,j-c-2)v^iw^jx^n\\
&+\sum_{n\geq3}\sum_{i=1}^{n-3}\sum_{j=i+3}^{n-1}\left(a_{n-1}(i,j)v^iw^jx^n
+\sum_{a=1}^{i-1}\sum_{c=0}^{i-a-1}\binom{i-a-1}{c}a_{n-c-2}(a,j-c-1)v^iw^jx^n\right)\\
&=wxB(x,v,w)+w^3xD(x,vw)\\
&+\sum_{a\geq1}\sum_{n\geq a+2}\sum_{j=a+2}^n\sum_{i=1}^{j-a-1}
a_n(a,j)\frac{v^{i+a}w^{j+2}x^{n+2}}{(1-vwx)^i}\\
&+\sum_{n\geq3}\sum_{i=1}^{n-3}\sum_{j=i+3}^{n-1}\left(a_{n-1}(i,j)v^iw^jx^n
+\sum_{a=1}^{i-1}\sum_{c=0}^{i-a-1}\binom{i-a-1}{c}a_{n-c-2}(a,j-c-1)v^iw^jx^n\right)\\
&=wxB(x,v,w)+w^3xD(x,vw)\\
&+\frac{w^2x^2(1-vwx)}{vwx+v-1}A^+(x,1-vwx,\frac{vw}{1-vwx})
-\frac{vw^2x^2}{vwx+v-1}A^+(x,v,w)\\
&+\sum_{n\geq3}\sum_{i=1}^{n-3}\sum_{j=i+3}^{n-1}\left(a_{n-1}(i,j)v^iw^jx^n
+\sum_{a=1}^{i-1}\sum_{c=0}^{i-a-1}\binom{i-a-1}{c}a_{n-c-2}(a,j-c-1)v^iw^jx^n\right)\\
&=wxB(x,v,w)+w^3xD(x,vw)\\
&+\frac{w^2x^2(1-vwx)}{vwx+v-1}A^+(x,1-vwx,\frac{vw}{1-vwx})
-\frac{vw^2x^2}{vwx+v-1}A^+(x,v,w)\\
&+xB(x,v,w) +\frac{1}{vwx+v-1}\sum_{a\geq1}\sum_{n\geq a+3}\sum_{j=a+3}^n
a_n(a,j)(\frac{v^{j-1}w^{j+1}x^{nn+2}}{(1-vwx)^{j-2-a}}
-v^{a+1}w^{j+1}x^{nn+2})\\
&=wxB(x,v,w)+w^3xD(x,vw)\\
&+\frac{w^2x^2(1-vwx)}{vwx+v-1}A^+(x,1-vwx,\frac{vw}{1-vwx}) -\frac{vw^2x^2}{vwx+v-1}A^+(x,v,w)\\
&+xB(x,v,w) +\frac{wx^2(1-vwx)^2}{v(vwx+v-1)}B(x,1-vwx,\frac{vw}{1-vwx})
-\frac{vwx^2}{vwx+v-1}B(x,v,w).
\end{align*}

Hence, one gets the following system of functional equations.
\begin{proposition}\label{pr1243x1342a}
We have
\begin{align*}
A^+(x,v,w)&=wC(x,vw)+w^2D(x,vw)+B(x,v,w),\\
A^-(x,v,w)&=v^2wx^2+\frac{vx}{1-v}A(x,vw,1)-\frac{v^2x}{1-v}A(vx,w,1),\\
C(x,vw)&=vwx^2A(vwx,1,1)+vwx^2+v^2w^2x^3+xC(x,vw)\\
&+\frac{vwx^2}{vwx+vw-1}A^+(\frac{vwx}{1-vwx},1-vwx,1)\\
&-\frac{(1-vwx)x^2}{vwx+vw-1}A^+(x,1-vwx,\frac{vw}{1-vwx}),\\
D(x,vw)&=x^2A(vwx,1,1)-v^2w^2x^4+x(C(x,vw)-vwx^2A(vwx,1,1))+xD(x,vw)\\ &+\frac{x^2(1-vwx)}{vwx+vw-1}A^+(\frac{vwx}{1-vwx},1-vwx,1)\\
&-\frac{x^2(1-vwx)^2}{vw(vwx+vw-1)}A^+(x,1-vwx,\frac{vw}{1-vwx})-x^2C(x,vw),\\
B(x,v,w)&=wxB(x,v,w)+w^3xD(x,vw)\\
&+\frac{w^2x^2(1-vwx)}{vwx+v-1}A^+(x,1-vwx,\frac{vw}{1-vwx}) -\frac{vw^2x^2}{vwx+v-1}A^+(x,v,w)\\
&+xB(x,v,w) +\frac{wx^2(1-vwx)^2}{v(vwx+v-1)}B(x,1-vwx,\frac{vw}{1-vwx})
-\frac{vwx^2}{vwx+v-1}B(x,v,w).
\end{align*}
\end{proposition}

By Proposition \ref{pr1243x1342a}, one may express $B(x,v,w)$ in terms the generating functions $A,C,D$ and $C,D$ in terms of $A$.  Substituting these relations into the last equation from Proposition \ref{pr1243x1342a}, one obtains
\begin{align}\label{sseq2}
&\frac{vx-wx-v-x+1}{vwx+v-1}A^+(x,v,w)\\
&=\frac{w^2x^2(vwx-v-1)A^+(\frac{vwx}{1-vwx},1-vwx,1)}{vwx+vw-1}\notag\\
&+\frac{vwx^2(w^2-1)(vwx-1)A^+(x,1-vwx,\frac{vw}{1-vwx})}{(vwx+v-1)(vwx+vw-1)}\notag\\
&+w^2x^2(vwx-v-1)A(vwx,1,1)+x^2vw^2(vw^2x^2-vwx-1).\notag
\end{align}
Replacing $w$ by $w/v$ in \eqref{sseq2}, we have
\begin{align}\label{sseq2a}
&\frac{v^2x-wx-v^2-vx+v}{v(wx+v-1)}A^+(x,v,w/v)\\
&=\frac{w^2x^2(wx-v-1)A^+(\frac{wx}{1-wx},1-wx,1)}{v^2(wx+w-1)}
+\frac{wx^2(w^2-v^2)(wx-1)A^+(x,1-wx,\frac{w}{1-wx})}{v^2(wx+v-1)(wx+w-1)}\notag\\
&+\frac{w^2x^2}{v^2}(wx-v-1)A(wx,1,1)+\frac{x^2w^2}{v^2}(w^2x^2-vwx-v).\notag
\end{align}

We now seek to determine a formula for $A^+(x,v,w)$ using \eqref{sseq2a}.  First note that by \eqref{introe2} and the main result from \cite{Kr}, we have
\begin{align}\label{sseq0}
A(x,1,1)=\frac{1-3x-\sqrt{1-6x+x^2}}{2},
\end{align}
which implies $A^+(x,v,1)$ may determined independently of $A^-(x,v,1)$ since $A$ occurs in \eqref{sseq2a} only through the $A(wx,1,1)$ term. Suppose for a moment that $A(x,v,1)$ is as in Corollary \ref{co1243x1423} below.  Then the second equation in Proposition \ref{pr1243x1342a} above at $w=1$, together with the fact $A(x,v,1)=A^+(x,v,1)+A^-(x,v,1)$, gives a linear system of equations in the quantities $A^+$ and $A^-$.  This yields
\begin{align}\label{sseq1}
A^+(x,v,1)=\frac{vx^2(1+v-vx)}{2(vx-v-2x+1)}\sqrt{1-6vx+v^2x^2}
-\frac{vx^2(v^2x^2-v^2x-4vx+3v-1)}{2(vx-v-2x+1)},
\end{align}
which we will assume for now to aid in solving \eqref{sseq2a}.

Then taking $v=v_0=\frac{1-x+\sqrt{1-2(1+2w)x+(1+4w)x^2}}{2(1-x)}$ in \eqref{sseq2a}, and using \eqref{sseq0} and \small\eqref{sseq1}, implies
\begin{align*}
&A^+(x,1-wx,\frac{w}{1-wx})\\
&=\frac{(1-x)(wx+w-1)}{4(wx-w-2x+1)(wx-1)}\sqrt{1-6wx+w^2x^2}\sqrt{(1-x)(1-x-4vwx)}\\
&+\frac{1-w+(2w^2-w-2)x-(w-1)(2w^2+4w+1)x^2+w(2w^2-2w-1)x^3}{4(wx-w-2x+1)(wx-1)}\sqrt{1-6wx+w^2x^2}\\
&-\frac{(wx+w-1)(wx^2-wx-3x+1)}{4(wx-w-2x+1)(wx-1)}\sqrt{(1-x)(1-x-4vwx)}\\
&+\frac{w-1-(3w^2-4)x+(8w^3+w^2-9w-3)x^2}{4(wx-w-2x+1)(wx-1)}\\
&+\frac{-w(2w^3+8w^2-9w-4)x^3+w^2(2w^2-2w-1)x^4}{4(wx-w-2x+1)(wx-1)}.
\end{align*}\normalsize
Substituting this expression into \eqref{sseq2}, and using \eqref{sseq0} and \eqref{sseq1}, we obtain
\begin{align}\label{pr1243x1342aAP}
&A^+(x,v,w)\\
&=\frac{(w^2-1)(1-x)vwx^2\sqrt{1-6vwx+v^2w^2x^2}\sqrt{(1-x)(1-x-4vwx)}}{4(vwx-vw-2x+1)(vx-wx-v-x+1)}\notag\\
&+\frac{(1-x-2vwx)vw(1-x)x^2\sqrt{1-6vwx+v^2w^2x^2}}
{4(vwx-vw-2x+1)(vx-wx-v-x+1)}\notag\\
&+\frac{(1-2v^2x^2+4v^2x-2v^2-x^2-2x-2vwx(1-x))vw^3x^2\sqrt{1-6vwx+v^2w^2x^2}}
{4(vwx-vw-2x+1)(vx-wx-v-x+1)}\notag\\
&+\frac{(1-w^2)(vwx^2-vwx-3x+1)vwx^2\sqrt{(1-x)(1-x-4vwx)}}{4(vwx-vw-2x+1)(vx-wx-v-x+1)}\notag\\
&+\frac{(3(w^2-1)x^2+4(1-2w)x-w^2+4w-1)vwx^2}{4(vwx-vw-2x+1)(vx-wx-v-x+1)}\notag\\
&-\frac{((w^2-1)x^3+8(w^2+1)x^2-(7w^2+4w+11)x+4w+4)v^2w^2x^2}{4(vwx-vw-2x+1)(vx-wx-v-x+1)}\notag\\
&-\frac{2((1-x)(w^2x^2+x^2+3x-3)+vwx(1-x)^2)v^3w^3x^2}{4(vwx-vw-2x+1)(vx-wx-v-x+1)}.\notag
\end{align}

One may verify using programming that this expression for $A^+(x,v,w)$ indeed satisfies \eqref{sseq2} and \eqref{sseq1} and thus is the solution of \eqref{sseq2} that is sought.  We may now determine $A^-(x,v,w)$.  By the second equation in Proposition \ref{pr1243x1342a}, we have
\begin{align*}
A^-(x,v,1)&=v^2x^2+\frac{vx}{1-v}A(x,v,1)-\frac{v^2x}{1-v}A(vx,1,1)\\
&=v^2x^2+\frac{vx}{1-v}(A^+(x,v,1)+A^-(x,v,1))-\frac{v^2x}{1-v}A(vx,1,1).
\end{align*}

Solving for $A^-(x,v,1)$ in this last equation, and using \eqref{sseq0} and \eqref{sseq1}, yields
$$A^-(x,v,1)=\frac{v^2x(x-1)(vx^2-vx-3x+1+(x-1)\sqrt{1-6vx+v^2x^2})}{2(vx-v-2x+1)}.$$
Hence, by Proposition \ref{pr1243x1342a}, we get the following explicit formula for $A^-(x,v,w)$:
\small\begin{align}\label{pr1243x1342aAM}
&A^-(x,v,w)\\
&=\frac{((1+v-2vx)x+(v^2x-v^2+vx-1)wx-(1-x)(1-vx)vw^2)v^2wx^2}{2(vwx-2vx-w+1)(vwx-vw-2x+1)}\sqrt{1-6vwx+v^2w^2x^2}\notag\\
&+\frac{(6vx^2-3(v+1)x+2-(2v^2x^3+2v(v+1)x^2-(3v^2+2v+3)x+2v+2)w)v^2wx^2}{2(vwx-2vx-w+1)(vwx-vw-2x+1)}\notag\\
&+\frac{(v^2x^3-v^2x^2+vx^3+3vx^2-3vx-x^2-3x+3-(1-x)(1-vx)vwx)v^3w^3x^2}{2(vwx-2vx-w+1)(vwx-vw-2x+1)}.\notag
\end{align}\normalsize

We then have the following formula for $A(x,v,w)$.

\begin{theorem}
The generating function for the joint distribution of the first and second letter statistics on $\mathcal{S}_n(1243,1423)$ for $n \geq 1$is given by
$$A^+(x,v,w)+A^-(x,v,w),$$
where $A^+(x,v,w)$ and $A^-(x,v,w)$ are as in \eqref{pr1243x1342aAP} and \eqref{pr1243x1342aAM}, respectively.
\end{theorem}

Substituting $w=1$ in the prior theorem and finding $vx+A(x,v,1)$, one obtains the following result.

\begin{corollary}\label{co1243x1423}
The generating function for the distribution of the first letter statistic on $\mathcal{S}_n(1243,1423)$ for $n \geq 1$ is given by
$$\frac{vx(2-3v-3x+3vx)+vx(v+x-vx)(vx+\sqrt{1-6vx+v^2x^2})}{2(1-v-2x+vx)}.$$
\end{corollary}

\subsection{The case $(1243,1342)$.}

We first write a recurrence for $a_n(i,j)$ when $i<j$.

\begin{lemma}\label{12431342L1}
If $1 \leq i <j \leq n-1$, then
\begin{align}\label{12431342L1e1}
a_n(i,j)&=\sum_{k=i+1}^{j-1}a_{n-1}(i,k)+\sum_{a=1}^{i-1}\sum_{c=0}^{i-a-1}\sum_{b=a+1}^{j-c-2}\binom{i-a-1}{c}a_{n-c-2}(a,b)\notag\\
&\quad+\delta_{i+1,j}\cdot\left(a_{n-1}(i,i+1)+\sum_{a=1}^{i-1}\sum_{c=0}^{i-a-1}\binom{i-a-1}{c}a_{n-c-2}(a,i-c)\right),
\end{align}
with $a_n(i,n)=a_{n-1}(i)$ for $1 \leq i \leq n-1$.
\end{lemma}
\begin{proof}
The formula when $j=n$ is obvious, so assume $j<n$.  Let $\mathcal{A}_{n,i,j}$ denote the subset of $\mathcal{S}_{n,i}(1243,1342)$ having second letter $j$. Let $\pi \in \mathcal{A}_{n,i,j}$ where $1 \leq i<j\leq n-1$ and $k$ denote the third letter of $\pi$.  Suppose $k<j$, noting that this is indeed a requirement if $j \geq i+2$, for otherwise $j<n$ would imply a $1342$ would occur.  If $i+1 \leq k \leq j-1$, then the letter $j$ is extraneous and thus may be deleted from $\pi$ since all elements of $[k+1,n]-\{j\}$ must occur in increasing order with all elements of $[i+1,k]$ occurring to the left of those in $[k+1,n]-\{j\}$.  This implies the first summation formula in \eqref{12431342L1e1}.  Otherwise $k<i$ and we may write $\pi=ijx_1\cdots x_cad\rho$, where $1\leq a<x_1<\cdots <x_c\leq i-1$ and $d>a$.  Note that $d<j$ if $j \geq i+2$ in order to avoid a $1342$ of the form $ijdt$ for some $t\in[i+1,j-1]$.

We argue now that $d<j$ implies that $i$ and $j$ may be deleted from $\pi$.  Clearly, the letters $i$ and $j$ are extraneous concerning the avoidance of $1243$ in light of the ascent $a,d$ where $a<i$.  They are also irrelevant concerning $1342$, whence they may be deleted.  To see this, note that if $d<i$, then members of $[d+1,n]$ occurring to the right of $d$ must form an increasing subsequence due to $a$ preceding $d$ and thus no $1342$ can start with $ij$.  The same conclusion is reached if $i<d<j$, for in this case all letters in $[i+1,d-1]$ occur to the left of those in $[d+1,n]-\{j\}$, with the latter forming an increasing subsequence. Since $x_1,\ldots,x_c$ may clearly also be deleted from $\pi$ as $a<x_c$, one is  left with $\pi'\in\mathcal{A}_{n-c-2,a,b}$ for some $a$ and $b$.  Note that $b \in [a+1,j-c-2]$ (after reducing letters) since $d\in[a+1,j-1]-\{x_1,\ldots,x_c,i\}$.  Considering all possible $a$, $c$ and $b$ then yields the triple sum expression in \eqref{12431342L1e1} and finishes the case when $j \geq i+2$.  On the other hand, if $j=i+1$, then $k=j+1$ is possible without introducing an occurrence of $1342$, in which case $k$ may be deleted resulting in a member of $\mathcal{A}_{n-1,i,i+1}$.  Likewise, $d=j+1$ is also possible in the decomposition of $\pi$ above.  Combining these two additional cases then accounts for the second line in formula \eqref{12431342L1e1} and completes the proof.
\end{proof}

From \eqref{12431342L1e1}, we may deduce the following further useful formulas.

\begin{lemma}\label{12431342L2}
If $1 \leq i \leq n-2$, then
\begin{equation}\label{12431342L2e0}
a_n(i,i+1)=a_n(i,i+2)
\end{equation}
and
\begin{equation}\label{12431342L2e1}
a_n(i,i+1)=\sum_{\ell=1}^ia_{n-1}(\ell,i+1).
\end{equation}
\end{lemma}
\begin{proof}
Both equalities are easily seen to hold if $i=n-2$, so assume $1 \leq i \leq n-3$.  Taking $j=i+1$ and $j=i+2$ in \eqref{12431342L1e1}, and comparing the results, then completes the proof of \eqref{12431342L2e0}.  For \eqref{12431342L2e1}, first observe
\begin{align*}
\sum_{\ell=1}^{i-1}a_{n-1}(\ell,i+1)&=\sum_{\ell=1}^{i-1}\sum_{k=\ell+1}^ia_{n-2}(\ell,k)+\sum_{\ell=1}^{i-1}\sum_{a=1}^{\ell-1}\sum_{c=0}^{\ell-a-1}\sum_{b=a+1}^{i-c-1}\binom{\ell-a-1}{c}a_{n-c-3}(a,b)\\
&=\sum_{\ell=1}^{i-1}\sum_{k=\ell+1}^ia_{n-2}(\ell,k)+\sum_{a=1}^{i-2}\sum_{c=0}^{i-a-2}\sum_{b=a+1}^{i-c-1}a_{n-c-3}(a,b)\sum_{\ell=a+c+1}^{i-1}\binom{\ell-a-1}{c}\\
&=\sum_{\ell=1}^{i-1}\sum_{k=\ell+1}^ia_{n-2}(\ell,k)+\sum_{a=1}^{i-2}\sum_{c=1}^{i-a-1}\sum_{b=a+1}^{i-c}\binom{i-a-1}{c}a_{n-c-2}(a,b).
\end{align*}
Then by \eqref{12431342L2e0} and \eqref{12431342L1e1} when $j=i+2$, we have
\begin{align*}
&a_n(i,i+1)-\sum_{\ell=1}^{i-1}a_{n-1}(\ell,i+1)=a_n(i,i+2)-\sum_{\ell=1}^{i-1}a_{n-1}(\ell,i+1)\\
&=a_{n-1}(i,i+1)-\sum_{\ell=1}^{i-1}\sum_{k=\ell+1}^ia_{n-2}(\ell,k)\\
&\quad+\left(\sum_{a=1}^{i-1}\sum_{c=0}^{i-a-1}\sum_{b=a+1}^{i-c}-\sum_{a=1}^{i-2}\sum_{c=1}^{i-a-1}\sum_{b=a+1}^{i-c}\right)\binom{i-a-1}{c}a_{n-c-2}(a,b)\\
&=a_{n-1}(i,i+1)-\sum_{\ell=1}^{i-1}\sum_{k=\ell+1}^ia_{n-2}(\ell,k)+\sum_{a=1}^{i-1}\sum_{b=a+1}^ia_{n-2}(a,b)=a_{n-1}(i,i+1),
\end{align*}
which completes the proof of \eqref{12431342L2e1}.
\end{proof}

To summarize, we have the following recurrence for $a_n(i,j)$:
\begin{align}\label{r1243x1342}
a_n(i,j)&=a_{n-1}(j),\quad 1\leq j<i\leq n,\notag\\
a_n(i,n)&=a_{n-1}(i),\quad 1\leq i\leq n-1,\notag\\
a_n(i,i+1)&=\sum_{\ell=1}^ia_{n-1}(\ell,i+1),\quad 1\leq i\leq n-2,\\
a_n(i,i+2)&=a_n(i,i+1), \quad  1\leq i\leq n-2,\notag\\
a_n(i,j)&=\sum_{k=i+1}^{j-1}a_{n-1}(i,k)+\sum_{a=1}^{i-1}\sum_{c=0}^{i-a-1}\sum_{b=a+1}^{j-c-2}\binom{i-a-1}{c}a_{n-c-2}(a,b),\notag
\end{align}
for $4 \leq i+3 \leq j \leq n-1$.

In this case, we state the recurrence formulas satisfied by the corresponding distribution polynomials which will aid in translating \eqref{r1243x1342} to functional equations as they are not too lengthy.  Define $A_n^+(v,w)=\sum_{i=1}^{n-1}\sum_{j=i+1}^na_n(i,j)v^iw^j$, $A_n^-(v,w)=\sum_{i=2}^{n}\sum_{j=1}^{i-1}a_n(i,j)v^iw^j$ and $A_n(v,w)=\sum_{i=1}^{n}\sum_{j=1}^na_n(i,j)v^iw^j$ for $n\geq2$. Clearly, $A_n(v,w)=A^+_n(v,w)+A^-_n(v,w)$ for all $n\geq2$. Further, we define
$$C_n(v)=\sum_{i=1}^{n-2}a_n(i,i+1)v^i \quad \text{and} \quad B_n(v,w)=\sum_{i=1}^{n-4}\sum_{j=i+3}^{n-1}a_n(i,j)v^iw^j.$$

Then \eqref{r1243x1342} may be rewritten in terms of these distributions as
\begin{align*}
A_n^-(v,w)&=\frac{v}{1-v}A_{n-1}(vw,1)-\frac{v^{n+1}}{1-v}A_{n-1}(w,1),\\
A_n^+(v,w)&=B_n(v,w)+w^2(C_n(vw)-(vw)^{n-2}A_{n-2}(1,1))+wC_n(vw)+w^nA_{n-1}(1,w),\\
C_n(v)&=\frac{1}{v}A^+_{n-1}(1,v),\\
B_n(v,w)&=\frac{w}{1-w}(B_{n-1}(v,w)+w^2(C_{n-1}(vw)-A_{n-3}(1,1)(vw)^{n-3}))\\
&-\frac{w^n}{1-w}(B_{n-1}(v,1)+C_{n-1}(v)-A_{n-3}(1,1)v^{n-3})\\
&+\frac{w^3}{1-w}(C_{n-1}(vw)-A_{n-3}(1,1)v^{n-3}w^{n-3}) -\frac{w^n}{1-w}(C_{n-1}(v)-A_{n-3}(1,1)v^{n-3})\\
&+\sum_{i=1}^{n-4}\sum_{j=i+3}^{n-1}\sum_{a=1}^{i-1}
\sum_{c=0}^{i-a-1}\sum_{b=a+1}^{j-c-2}\binom{i-1-a}{c}a_{n-c-2}(a,b)v^iw^j.
\end{align*}

Define $A^\pm(x,v,w)=\sum_{n\geq2}A_n^\pm(v,w)x^n$ and $A(x,v,w)=\sum_{n\geq2}A_n(v,w)x^n$.  Further, define $B(x,v,w)=\sum_{n\geq2}B_n(v,w)x^n$ and $C(x,v)=\sum_{n\geq2}C_n(v)x^n$.  Rewriting the preceding recurrences in terms of generating functions yields the following result.

\begin{proposition}\label{pr1}
We have
\begin{align*}
A^-(x,v,w)&=v^2wx^2+\frac{vx}{1-v}A(x,vw,1)-\frac{v^2x}{1-v}A(vx,w,1),\\
A^+(x,v,w)&=vw^2x^2-vw^3x^3+B(x,v,w)+w^2(C(x,vw)-x^2A(vwx,1,1))\\
&+wC(x,vw)+wxA(wx,v,1),
\end{align*}
where
\begin{align*}
C(x,v)&=\frac{x}{v}A^+(x,1,v),\\
B(x,v,w)&=\frac{wx(1-v)}{(1-v-vwx)(1-w)}(B(x,v,w)-B(wx,v,1))\\
&+\frac{2(1-vw)w^3x}{(1-vw-vwx)(1-w)}C(x,vw)-\frac{2wx(1-v)}{(1-v-vwx)(1-w)}C(wx,v)\\
&+\frac{x^2w^2(1-vwx)^2}{(1-v-vwx)(1-vw-vwx)}
(B(\frac{vwx}{1-vwx},1-vwx,1)-B(x,1-vwx,\frac{vw}{1-vwx}))\\
&+\frac{2x^2w^2(1-vwx)^2}{(1-v-vwx)(1-vw-vwx)}C(\frac{vwx}{1-vwx},1-vwx).
\end{align*}
\end{proposition}

By mathematical programming, one may verify the following solution of the foregoing system of functional equations.

\begin{theorem}\label{th1243x1342}
The generating function for the joint distribution of the first and second letter statistics on $\mathcal{S}_n(1243,1342)$ for $n \geq 1$is given by $A(x,v,w)=A^+(x,v,w)+A^-(x,v,w)$, where
\begin{align*}
&A^+(x,v,w)\\
&=\frac{(1-v+(v^2x-2v^2+vx-x)w-vx(x-2)(v-1)w^2)vw^2x^2}{2(1-v+wx(v-2))(1-2vw-x(1-vw))}\sqrt{v^2w^2x^2-6vwx+1}\\
&+\frac{(1-v+(6v^2-4v-x(4v^2-2v+1))w+xv(x(v^2+4v-4)-2v^2-6v+6)w^2)vw^2x^2}{2(1-v+wx(v-2))(1-2vw-x(1-vw))}\\
&-\frac{(x-2)(v-1)v^3w^5x^4}{2(1-v+wx(v-2))(1-2vw-x(1-vw))}\\
&=vw^2x^2+w^2v(vw+w+1)x^3+w^2v(2v^2w^2+2vw^2+2vw+2w^2+w+1)x^4+\cdots
\end{align*}
and
\begin{align*}
&A^-(x,v,w)\\
&=\frac{(x(1+v-2vx)+x(v^2x-v^2+vx-1)w-v(x-1)(vx-1)w^2)v^2wx^2}{2(1-vw-x(2-vw))(1-w-vx(2-w))}\sqrt{v^2w^2x^2-6vwx+1}\\
&+\frac{((3x-2)(w-1)+(-w^2x^2-3w^2x-2wx^2+3w^2+2wx+6x^2-2w-3x)v)v^2wx^2}{2(1-vw-x(2-vw))(1-w-vx(2-w))}\\
&+\frac{(w^2x+wx^2-w^2+3wx-2x^2-3w-2x+3-(x-1)(w-1)vwx)v^4w^2x^3}{2(1-vw-x(2-vw))(1-w-vx(2-w))}\\
&=v^2wx^2+(vw+v+1)v^2wx^3+2(v^2w^2+v^2w+v^2+vw+v+1)v^2wx^4+\cdots.
\end{align*}
Moreover, the generating functions counting the members of $\mathcal{S}_{n}(1243,1342)$ starting with an ascent of size greater than two or of size exactly one and whose second letter is not $n$ in either case according to the first and second letter statistics are given respectively by
\begin{align*}
B(x,v,w)&=\frac{(vwx-2wx^2+2wx+x-1)w^2x^2}{2(1-2vw-x(1-vw))(1-v-wx(2-v))}\sqrt{v^2w^2x^2-6vwx+1}\\
&+\frac{(1-x+(3vx-4v+2x-2)wx-(2x+v-6)vw^2x^2)w^2x^2}{2(1-2vw-x(1-vw))(1-v-wx(2-v))}\\
&=2vw^4x^5+2(4vw+2w+1)vw^4x^6+2\big(17v^2w^2+10vw^2+5vw+4w^2\\
&\quad+2w+1\big)vw^4x^7+\cdots
\end{align*}
and
\begin{align*}
C(x,v)&=\frac{-vx^2(vx^2-2vx-3x+2+(x-2)\sqrt{v^2x^2-6vx+1})}{2(1-x-2v+vx)}\\
&=vx^3+(2v+1)vx^4+(6v^2+3v+1)vx^5+(22v^3+11v^2+4v+1)vx^6+\cdots.
\end{align*}
\end{theorem}

Substituting $w=1$ in the prior theorem and finding $vx+A(x,v,1)$, one obtains the following result.

\begin{corollary}\label{co1243x1342}
The generating function for the distribution of the first letter statistic on $\mathcal{S}_n(1243,1342)$ for $n \geq 1$ is given by
$$\frac{vx(2-3v-3x+3vx)+vx(v+x-vx)(vx+\sqrt{1-6vx+v^2x^2})}{2(1-v-2x+vx)}.$$
\end{corollary}

\subsection{The case $(1243,1324)$.}

We first write a recurrence for $a_n(i,j)$ when $i<j$.

\begin{lemma}\label{1243,1324}
We have
\begin{equation}\label{i(i+1)rec}
a_n(i,i+1)=a_{n-1}(i,i+1)+\sum_{a=1}^{i-1}\sum_{c=0}^{i-a-1}\sum_{b=a+1}^{i-c}\binom{i-a-1}{c}a_{n-c-2}(a,b),\quad 1 \leq i \leq n-2,
\end{equation}
and
\begin{equation}\label{ijrec}
a_n(i,j)=a_{n-1}(i,j)+\sum_{a=1}^{i-1}\sum_{c=0}^{i-a-1}\binom{i-a-1}{c}a_{n-c-2}(a,j-c-1), \quad 3 \leq i+2\leq j \leq n-1,
\end{equation}
with $a_n(i,n)=a_{n-1}(i)$ for $1 \leq i \leq n-1$.
\end{lemma}
\begin{proof}
A similar proof may be given as in the prior two cases.  Note that the formula for $a_n(i,n)$ is obvious since an $n$ in the second position may clearly be removed.  Let $k$ denote the third letter of a member of $\mathcal{S}_{n,i}(1243,1324)$. If $j=i+1$ where $1 \leq i \leq n-2$, then either $k=i+2$ or $k<i$, where clearly $k$ may be deleted in the former case.  Assuming the latter, let $a,d$ denote the second leftmost ascent.  Then we must have $a+1\leq d \leq i+2$, for otherwise a $1243$ would occur.  Thus, each letter prior to $a$ may be deleted in this case and considering all possible $a$, $b$ and $c$, where $b$ and $c$ are as before, implies formula \eqref{i(i+1)rec}.  If $i+2 \leq j <n$, then we have $k=j+1$ or $k<i$, the former leading to $a_{n-1}(i,j)$ possibilities.  On the other hand, if $k<i$ and $a,d$ denotes the second ascent, then we must have $d=j+1$.  To see this, note that elements of $[d+1,n]$ to the right of $d$ must form an increasing subsequence and thus $d<j<n$ would imply an occurrence of $1324$ of the form $ijxn$ where $x\in[i+1,j-1]$.  Therefore $d=j+1$ implies all letters prior to $a$ again may be deleted, resulting in a permutation that starts $a,j-c-1$.  Considering all possible $a$ and $c$ then accounts for the double sum expression in \eqref{ijrec} and completes the proof.
\end{proof}

From this, one may deduce the following further useful formula.

\begin{lemma}\label{L1ij}
If $1 \leq i \leq n-3$ and $i+2 \leq j \leq n-1$, then
\begin{equation}\label{L1ije1}
a_n(i,j)=a_{n-1}(1,j-1)+a_{n-1}(2,j-1)+\cdots+a_{n-1}(i-1,j-1)+a_{n-1}(i,j).
\end{equation}
\end{lemma}
\begin{proof}
Note that \eqref{L1ije1} is clearly true on combinatorial grounds if $i=1$ since the third letter of a member of $\mathcal{S}_{n}(1243,1324)$ starting with $1,j$ where $j<n$ must be $j+1$.  So let $i \geq 2$ and $j \in [i+2,n-1]$.  Then by \eqref{ijrec}, we have
\begin{equation}\label{L1ije2}
\sum_{\ell=1}^{i-1}a_{n-1}(\ell,j-1)=\sum_{\ell=1}^{i-1}\left(a_{n-2}(\ell,j-1)+\sum_{a=1}^{\ell-1}\sum_{c=0}^{\ell-a-1}\binom{\ell-a-1}{c}a_{n-c-3}(a,j-c-2)\right).
\end{equation}
If $i \geq 3$, then
\begin{align*}
&\sum_{\ell=2}^{i-1}\sum_{a=1}^{\ell-1}\sum_{c=0}^{\ell-a-1}\binom{\ell-a-1}{c}a_{n-c-3}(a,j-c-2)\\
&=\sum_{a=1}^{i-2}\sum_{c=0}^{i-a-2}a_{n-c-3}(a,j-c-2)\sum_{\ell=a+c+1}^{i-1}\binom{\ell-a-1}{c}\\
&=\sum_{a=1}^{i-2}\sum_{c=0}^{i-a-2}a_{n-c-3}(a,j-c-2)\binom{i-a-1}{c+1}\\
&=\sum_{c=1}^{i-2}\sum_{a=1}^{i-c-1}\binom{i-a-1}{c}a_{n-c-2}(a,j-c-1).
\end{align*}
Thus, the right-hand side of \eqref{L1ije2} is given by
\begin{align*}
&\sum_{\ell=1}^{i-1}a_{n-2}(\ell,j-1)+\sum_{c=1}^{i-2}\sum_{a=1}^{i-c-1}\binom{i-a-1}{c}a_{n-c-2}(a,j-c-1)\\
&=\sum_{c=0}^{i-2}\sum_{a=1}^{i-c-1}\binom{i-a-1}{c}a_{n-c-2}(a,j-c-1)\\
&=\sum_{a=1}^{i-1}\sum_{c=0}^{i-a-1}\binom{i-a-1}{c}a_{n-c-2}(a,j-c-1)\\
&=a_n(i,j)-a_{n-1}(i,j),
\end{align*}
again by \eqref{ijrec}, which completes the proof.
\end{proof}

Summarizing, we have the following recurrence relations satisfied by $a_n(i,j)$:
\begin{align}\label{eqms1x1}
\left\{\begin{array}{ll}
a_n(i,j)&=a_{n-1}(j),\quad 1\leq j<i\leq n,\\[4pt]
a_n(i,n)&=a_{n-1}(i), \quad 1 \leq i \leq n-1,\\
a_n(i,j)&=\sum_{\ell=1}^{i-1}a_{n-1}(\ell,j-1)+a_{n-1}(i,j),\\
&\qquad\qquad\qquad\qquad\qquad\qquad\quad1 \leq i \leq n-3\mbox{ and }i+2 \leq j \leq n-1,\\[4pt]
a_n(i,i+1)&=a_{n-1}(i,i+1)+\sum_{a=1}^{i-1}\sum_{b=a+1}^i\sum_{c=0}^{i-b}\binom{i-a-1}{c}a_{n-c-2}(a,b),\\
&\qquad\qquad\qquad\qquad\qquad\qquad\quad 1 \leq i \leq n-2.
\end{array}
\right.
\end{align}

Define $A_n(v,w)$, $A_n^{\pm}(v,w)$ and $C_n(v)$ as in the previous subsection, but with $B_n(v,w)$ now given by $B_n(v,w)=\sum_{i=1}^{n-3}\sum_{j=i+2}^{n-1}a_n(i,j)v^iw^j$.
Translating \eqref{eqms1x1} then yields the following recurrences:
\begin{align*}
A_n^-(v,w)&=\sum_{i=2}^{n}\sum_{j=1}^{i-1}a_{n-1}(j)v^iw^j=\frac{v}{1-v}A_{n-1}(vw,1)-\frac{v^{n+1}}{1-v}A_{n-1}(w,1),\\
A_n^+(v,w)&=\sum_{i=1}^{n-2}\sum_{j=i+1}^{n-1}a_n(i,j)v^iw^j+\sum_{i=1}^{n-1}a_n(i,j)v^iw^n\\
&=\sum_{i=1}^{n-3}\sum_{j=i+2}^{n-1}a_n(i,j)v^iw^j+w\sum_{i=1}^{n-2}a_n(i,i+1)(vw)^i+w^nA_{n-1}(v,1)\\
&=B_n(v,w)+wC_n(vw)+w^nA_{n-1}(v,1),
\end{align*}
with
\begin{align*}
B_n(v,w)&=\sum_{i=2}^{n-3}\sum_{j=i+2}^{n-1}\sum_{k=1}^{i-1}a_{n-1}(i,j-1)v^iw^j
+\sum_{i=1}^{n-3}\sum_{j=i+2}^{n-1}a_{n-1}(i,j)v^iw^j\\
&=w\sum_{k=1}^{n-4}\sum_{j=k+3}^{n-1}\sum_{i=k+1}^{j-2}a_{n-1}(i,j-1)v^iw^j
+\sum_{i=1}^{n-4}\sum_{j=i+2}^{n-2}a_{n-1}(i,j)v^iw^j+\sum_{i=1}^{n-3}a_{n-2}(i)v^iw^{n-1}\\
&=\frac{wv}{1-v}B_{n-1}(v,w)-\frac{w}{1-v}B_{n-1}(1,vw)+B_{n-1}(v,w)+w^{n-1}A_{n-2}(v,1)\\
&-w^{n-1}v^{n-2}A_{n-3}(1,1)
\end{align*}
and
\begin{align*}
C_n(v)&=C_{n-1}(v)+v^{n-2}A_{n-3}(1,1)+\sum_{i=1}^{n-2}\sum_{a=1}^{i-1}\sum_{b=a+1}^i\sum_{c=0}^{i-b}\binom{i-a-1}{c}a_{n-2-c}(a,b)v^i\\
&=C_{n-1}(v)+v^{n-2}A_{n-3}(1,1)+\sum_{a=1}^{n-3}\sum_{b=a+1}^{n-2}\sum_{c=0}^{n-2-b}\sum_{i=b+c}^{n-2}\binom{i-a-1}{c}a_{n-2-c}(a,b)v^i.
\end{align*}

Define $A(x,v,w)$, $A^{\pm}(x,v,w)$ and $C(x,v)$ as in the previous subsection, with $B(x,v,w)=\sum_{n\geq2}B_n(v,w)x^n$ per the new definition for $B_n(v,w)$.  From the preceding recurrences, we obtain the following system of functional equations.

\begin{proposition}\label{pr1}
We have
\begin{align*}
A^-(x,v,w)&=v^2wx^2+\frac{vx}{1-v}A(x,vw,1)-\frac{v^2x}{1-v}A(vx,w,1),\\
A^+(x,v,w)&=vw^2x^2+B(x,v,w)+wC(x,vw)+wxA(wx,v,1),
\end{align*}
where
\begin{align*}
B(x,v,w)&=\frac{wx}{1-v}(vB(x,v,w)-B(x,1,vw))+xB(x,v,w)+wx^2A(wx,v,1)\\
&-vw^2x^3A(vwx,1,1)-v^2w^3x^4,\\
C(x,v)&=\frac{vx^3(1+vx)}{1-x}+\frac{vx^3}{1-x}A(vx,1,1)\\
&+\frac{x^2}{(vx+v-1)(1-x)}(vA^+(\frac{vx}{1-vx},1-vx,1)
-(1-vx)A^+(x,1-vx,\frac{v}{1-vx})).
\end{align*}
\end{proposition}

By programming, one may verify that the solution of the foregoing system is given as follows.

\begin{theorem}
We have $A(x,v,w)=A^+(x,v,w)+A^-(x,v,w)$, where
\begin{align*}
&A^+(x,v,w)\\
&=\frac{(1-v+(v^2x-2v^2+vx-x)w-vx(x-2)(v-1)w^2)vw^2x^2}{2(1-v+wx(v-2))(1-2vw-x(1-vw))}\sqrt{v^2w^2x^2-6vwx+1}\\
&+\frac{(1-v+(6v^2-4v-x(4v^2-2v+1))w+xv(x(v^2+4v-4)-2v^2-6v+6)w^2)vw^2x^2}{2(1-v+wx(v-2))(1-2vw-x(1-vw))}\\
&-\frac{(x-2)(v-1)v^3w^5x^4}{2(1-v+wx(v-2))(1-2vw-x(1-vw))}\\
&=vw^2x^2+w^2v(vw+w+1)x^3+w^2v(2v^2w^2+2vw^2+2vw+2w^2+w+1)x^4+\cdots
\end{align*}
and
\begin{align*}
&A^-(x,v,w)\\
&=\frac{(x(1+v-2vx)+x(v^2x-v^2+vx-1)w-v(x-1)(vx-1)w^2)v^2wx^2}{2(1-vw-x(2-vw))(1-w-vx(2-w))}\sqrt{v^2w^2x^2-6vwx+1}\\
&+\frac{((3x-2)(w-1)+(-w^2x^2-3w^2x-2wx^2+3w^2+2wx+6x^2-2w-3x)v)v^2wx^2}{2(1-vw-x(2-vw))(1-w-vx(2-w))}\\
&+\frac{(w^2x+wx^2-w^2+3wx-2x^2-3w-2x+3-(x-1)(w-1)vwx)v^4w^2x^3}{2(1-vw-x(2-vw))(1-w-vx(2-w))}\\
&=v^2wx^2+(vw+v+1)v^2wx^3+2(v^2w^2+v^2w+v^2+vw+v+1)v^2wx^4+\cdots.
\end{align*}
Moreover,
\begin{align*}
&B(x,v,w)\\
&=-\frac{vw^2x^2(wx^2-2wx-x+1)}{2(vwx-2vw-x+1)(vwx-2wx-v+1)}\sqrt{v^2w^2x^2-6vwx+1}\\
&-\frac{vw^2x^2(vw^2x^3-2vw^2x^2-vwx^2+3vwx-3wx^2+2wx+x-1)}{2(vwx-2vw-x+1)(vwx-2wx-v+1)}\\
&=vw^3x^4+w^3v(3vw+2w+1)x^5+w^3v(11v^2w^2+8vw^2+4vw+4w^2+2w+1)x^6+\cdots
\end{align*}
and
\begin{align*}
C(x,v)=\frac{-vx^2(vx^2-2vx-3x+2+(x-2)\sqrt{v^2x^2-6vx+1})}{2(1-x-2v+vx)}.
\end{align*}
\end{theorem}

Taking $w=1$ in the prior theorem yields the following result.

\begin{corollary}\label{co1243x1324}
The generating function for the distribution of the first letter statistic on $\mathcal{S}_n(1243,1324)$ for $n \geq 1$ is given by
$$\frac{vx(2-3v-3x+3vx)+vx(v+x-vx)(vx+\sqrt{1-6vx+v^2x^2})}{2(1-v-2x+vx)}.$$
\end{corollary}

\section{Conclusion}

In this paper, we have shown that each of the nine classes in Conjecture 1 above has last letter statistic distribution given by the array $S_{n,k}$, making use of a variety of techniques.  In most of these cases, a more general result is established from which the desired equivalence follows as a specific case.  We have also demonstrated by numerical evidence that there are no other avoidance classes for which the conjecture applies.  We remark that concerning the final six pattern pairs in Conjecture 1, we have shown further that the ascents and last letter statistics has the same joint distribution on the avoidance class in question as does the dist and last letter (plus one) statistics on the class $I_n(\geq,-,>)$ of inversion sequences, where dist records the number of distinct positive letters in an inversion sequence.  This confirms a further conjecture of Lin and Kim (see \cite[Conjecture~3.3]{LK}).  In these cases, we count equivalently, by the number of descents, members of $\mathcal{S}_{n,i}(\sigma,\tau)$, where $\sigma,\tau$ denotes (the reversal of) one of the last six pattern pairs in Conjecture 1.  The arguments used to establish these further $q$-results, where $q$ marks the number of descents, are however quite technical and differ when $q=1$ from the arguments presented here in the respective cases.  Details in the proof of these further results will be given in a forthcoming paper.

\end{document}